\documentclass[12pt]{amsart}
\usepackage{amsmath,amscd,amssymb,amsfonts}
\usepackage[all]{xy}

\newcommand{\ver}{\today}
\newcommand{\bC}{{\mathbb C}}

\newcommand{\bN}{{\mathbb N}}
\newcommand{\bP}{{\mathbb P}}

\newcommand{\bQ}{{\mathbb Q}}
\newcommand{\bZ}{{\mathbb Z}}

\newcommand{\cF}{{\mathcal F}}

\newcommand{\cS}{{\mathcal S}}

\newcommand{\codim}{\hbox{\rm codim}\,}

\newcommand{\Z}{\mathbb{Z}}
\newcommand{\M}{M_{f,0}}

\newcommand{\Q}{\mathbb{Q}}

\newcommand{\C}{\mathbb{C}}

\newcommand{\E}{\mathcal{E}}
\newcommand{\Sf}{\mathcal{O}}
\newcommand{\X}{\chi}

\theoremstyle{plain}
\newtheorem{thm}{Theorem}[section]
\newtheorem{cor}[thm]{Corollary}

\newtheorem{lem}[thm]{Lemma}
\newtheorem{prop}[thm]{Proposition}

\theoremstyle{definition}

\newtheorem{rem}[thm]{Remark}

\title[Spectrum of hyperplane arrangements in 4 variables]{Spectrum of hyperplane arrangements in four variables}
\author{Youngho Yoon}
\address{Center for Geometry and Physics, Institute for Basic Science (IBS), 77 Cheongam-ro, Nam-gu, Pohang, Gyeongbuk, Korea 790-784} \email{mathyyoon@ibs.re.kr}

\date{\ver}

\subjclass[2000]{Singularity}

\begin{document}

\maketitle

\begin{abstract}
One of the most important invariants in singularity theory is the Hodge spectrum. Calculating the Hodge spectrum is a difficult task and formulas exist for only a few cases. In this article the main result is the formula for reduced hyperplane arrangements in four variables. 

\end{abstract}

\section{Introduction}
The spectrum $Sp(f)$ of the germ of a hypersurface singularity $f:(\bC^n,\ 0)\rightarrow (\bC,\ 0)$ is a fractional Laurent polynomial $$Sp(f)=\sum_{\alpha\in\bQ} n_{f,\alpha} t^\alpha$$ with $n_{f,\alpha}\in \bZ$ which is defined from the Hodge filtration and the monodromy on the cohomology of the Milnor fiber of $f$ (see Section \ref{multiplicity}).

The spectra of generic hyperplane arrangements are known for $\alpha\in\bZ$ (see \cite{Sa8}-5.6). If the hyperplane arrangement germ is reduced with $n=2$,  it is easy to calculate. In \cite{Bu-S}, the reduced hyperplane arrangement case with $n=3$ is calculated and the case of $n=4$ is partially calculated. Here we give a full formula for the $n=4$ case. Beside the case of quasi-ordinary hypersurfaces (see \cite{P-V}), to our knowledge, this is the only other case of a complete Hodge spectrum computation for a large class of hypersurface singularities with the dimension of singular locus bigger than 1 in which we cannot use the Thom-Sebastiani formula.

We may assume $f$ is \emph{central} (i.e. $f$ has only linear forms as factors) since the spectrum $Sp(f)$ is defined locally. 

\begin{thm}\label{n4}
Assume  $f$ is a reduced central hyperplane arrangement with $d$ irreducible components in $\bC^4$. Let $m_{V}$ be the number of hyperplanes which pass through the edge $V$. Let $\cS$ be the set of dense edges excepting the hyperplanes in the arrangement.
Then we have the following formulas for $i\in \{1,\cdots,d\}$,
$$n_{f,\frac{i}{d}}= \eta_{0,i} (\langle\lceil im_{V}/d \rceil-1\rangle_{V\in \cS})\text{ and}$$
$$n_{f,1+\frac{i}{d}}=\eta_{1,i} (\langle\lceil im_{V}/d \rceil-1,\lfloor (d-i)m_{V}/d \rfloor\rangle_{V\in \cS}).$$
Similarly for $i\in \{0,\cdots,d-1\}$,
$$n_{f,4-\frac{i}{d}}=\eta_{0,i}(\langle\lfloor im_{V}/d \rfloor\rangle_{V\in \cS})\text{ and}$$
$$n_{f,3-\frac{i}{d}}=\eta_{1,i}( \langle\lfloor im_{V}/d \rfloor,\lceil (d-i)m_{V}/d \rceil-1\rangle_{V\in \cS}).$$
Otherwise $n_{f,\alpha}=0$. Here the functions $\eta_{0,i}$ and $\eta_{1,i}$ for each $i$ are defined as 
\begin{align*}
\eta_{0,i}(\langle u_{V}\rangle_{V\in \cS})&=\binom{i-1}{3}-\sum_{W\in\cS^{(3)}} \binom{u_{W}}{3}\\
&-\sum_{V\in\cS^{(2)}}\left( (i-3)\binom{u_{V}}{2}-2\binom{u_{V}}{3}\right)\\
&-\sum_{V\in\cS^{(2)}} \sum_{\substack{W\subset V\\ W\in \cS^{(3)}}}\left( 2 \binom{u_{V}}{3} -(u_{W}-2)\binom{u_{V}}{2}\right)+\delta_{0,i}
\end{align*}
and
\begin{align*}
\eta_{1,i}(\langle u_{V},v_{V}\rangle_{V\in \cS})&=(d-i-1)\binom{i-1}{2}-\sum_{W\in\cS^{(3)}}v_{W}\binom{u_{W}}{2}\\
&-\sum_{V\in\cS^{(2)}}\left(u_{V}v_{V}(i-2) +(d-i-1-2v_{V})\binom{u_{V}}{2}\right)\\
&+\sum_{V\in\cS^{(2)}}\sum_{\substack{W\subset V\\ W\in \cS^{(3)}}}\left( u_{V}v_{V}(u_{W}-u_{V})+v_{W}\binom{u_{V}}{2}\right),\\
\end{align*}
\noindent
where the set $\cS^{(k)}$ is the set of the codimension $k$ edges in $\cS$. The notations $\lfloor\cdot \rfloor$ and $\lceil\cdot\rceil$ mean floor and ceiling respectively. Also, $\delta_{0,i}=1$ if $0=i$ and $0$ otherwise.

\end{thm}
In this paper, we will follow the convention, $\binom{t}{k}=t(t-1)\cdots(t-k+1)/k!$ for $k\in \bN$ and  any $t$. For the definitions of edge and dense edge see Section \ref{hyper}. In this theorem, $\cS$ can be replaced by any set of edges containing all the dense edges with codimension $\geq 2$ (see Section \ref{cS}). In \cite{Bu-S}, the formula for $n_{f,\frac{i}{d}}$ for $i\in \{1,\cdots,d\}$ had been proved in the setting that $\cS$ is the set of edges in the non normal crossing singular locus of $f$ (see Section \ref{hyper}). 

If $f$ is \emph{not essential} (i.e. $f$ is a function of fewer variables for a possibly different choice of coordinates), we can apply the Thom-Sebastiani formula (see \cite{Kul}-II (8.10.6)) and recover the formulas for $n=3$ and $n=2$.

\begin{cor}\label{n3}
Assume  $f$ is a reduced central hyperplane arrangement with $d$ irreducible components in $\bC^3$. Let $m_{V}$ be the number of hyperplanes which pass through the edge $V$. Let $\cS$ be the set of codimension 2 dense edges of the arrangement.
Then we have the following formulas for $i\in \{1,\cdots,d\}$:
\begin{align*}
n_{f,\frac{i}{d}}=& \binom{i-1}{2}-\sum_{V\in\cS}\binom{\lceil im_V/d\rceil -1}{2},\\
n_{f,1+\frac{i}{d}}&=(i-1)(d-i-1)-\sum_{V\in\cS}(\lceil im_V/d\rceil -1)(m_V-\lceil im_V/d\rceil),\text{ and}\\
n_{f,2+\frac{i}{d}}&=\binom{d-i-1}{2}-\sum_{V\in\cS}\binom{m_V-\lceil im_V/d\rceil}{2}-\delta_{i,d}.
\end{align*}
Otherwise $n_{f,\alpha}=0$.
\end{cor}

This had been proved in  \cite{Bu-S}.

\begin{cor}\label{n2}
Assume  $f$ is a reduced central hyperplane arrangement with $d$ irreducible components in $\bC^2$. 
Then we have the following formulas for $i\in \{1,\cdots,d\}$
\begin{align*}
n_{f,\frac{i}{d}}=& i-1\text{ and }n_{f,1+\frac{i}{d}}=d-i-1+\delta_{i,d}.
\end{align*}
Otherwise $n_{f,\alpha}=0$.
\end{cor}

For the generic hyperplane arrangements in the $n=4$ case, we do not have codimension 2 and 3 dense edges. Thus, we get the following result.
\begin{cor}\label{generic}
Assume  $f$ is a generic hyperplane arrangement with $d$ irreducible components in $\bC^4$. 
Then we have the following formulas for $i\in \{1,\cdots,d\}$
$$n_{f,\frac{i}{d}}= \binom{i-1}{3} \text{ and } n_{f,1+\frac{i}{d}}=(d-i-1)\binom{i-1}{2}.$$
For $i\in \{0,\cdots,d-1\}$, we have
$$n_{f,4-\frac{i}{d}}=\binom{i-1}{3}+\delta_{0,i} \text{ and } n_{f,3-\frac{i}{d}}=(d-i-1)\binom{i-1}{2}.$$
Otherwise $n_{f,\alpha}=0$. Here $\delta_{0,i}=1$ if $0=i$ and $0$ otherwise.
\end{cor}

Consider \emph{decomposable} cases (i.e. after possibly a different change of coordinates, $f=f_1 f_2$ for two non-constant polynomials $f_1$ and $f_2$ in disjoint sets of variables).  According to \cite{Da}-Theorem 1.2, we know $n_{f,\alpha}=0$ except for $\alpha\in\bZ$ when the degrees of $f_1$ and $f_2$ are relatively prime. Here we give formulas for $\alpha\in\bZ$ in 4 variables.

\begin{cor}\label{3:1}
Assume $f(x_1, x_2, x_3, x_4)=f_1(x_1, x_2 ,x_3) f_2(x_4)$ for non-constant $f_1$ and $f_2$. Also, assume $f$ is a reduced central hyperplane arrangement with $d$ irreducible components in $\bC^4$. Then we have
\begin{align*}
n_{f,1}&=\binom{d-2}{2}-\sum_{V\in\cS^{(2)}}\binom{m_V-1}{2},\\
n_{f,2}&=-\binom{d-1}{2}+\sum_{V\in\cS^{(2)}}\binom{m_V-1}{2},\text{ and}\\ 
n_{f,3}&=d-1.
\end{align*}
Otherwise $n_{f,\alpha}=0$.
\end{cor}

\begin{cor}\label{2:2}
Assume $f(x_1, x_2, x_3, x_4)=f_1(x_1, x_2) f_2(x_3, x_4)$ for non-constant $f_1$ and $f_2$. Also, assume $f$ is a reduced central hyperplane arrangement with $d$ irreducible components in $\bC^4$. Let $s_1$ and $s_2$ be the degrees of $f_1$ and $f_2$ respectively. If $gcd(s_1,s_2)=1$ then we have the following formulas for $i\in \{1,\cdots,d\}$:
$$n_{f,1}=(s_1-1)(s_2-1),\ n_{f,2}=1-s_1 s_2,\text{ and } n_{f,3}=s_1 + s_2-1.$$
Otherwise $n_{f,\alpha}=0$.
\end{cor}

\noindent
{\bf Acknowledgements}: The author would like to thank his advisor, Nero Budur, as well as Melissa Davidson and Gabriel C. Drummond-Cole for constructive conversations about this paper and the Department of Mathematics at the University of Notre Dame for supporting him in his doctoral studies. He also thanks Alexandru Dimca for a correction and useful comments. This work was partially supported by the IBS (CA1305-02).

\section{Spectrum}

\subsection{Milnor fiber and Hodge spectrum}\label{multiplicity}
Let $f: (\bC^n,0) \rightarrow (\bC,0)$ be the germ of a non-zero holomorphic function. Then the Milnor fiber $M_{f,0}$ is defined as 
$$M_{f,0}=\{z\in \bC^n||z|<\epsilon \; and\; f(z)=t\} \text{ for } 0<|t|\ll\epsilon\ll 1.$$ 

The cohomology groups $H^*(M_f,\bC)$ carry canonical mixed Hodge structures such that the semi-simple part $T_s$ of the monodromy acts as an automorphism of finite order of these mixed Hodge structures (see \cite{St}-12.1.3). The eigenvalues $\lambda$ of the monodromy action on $H^*(M_f,\bC)$ are roots of unity. We define the {\it spectrum multiplicity} of $f$ at $\alpha\in \Q$ to be 
$$n_{f,\alpha}=\sum_{j\in \Z} (-1)^{j-n+1}\dim Gr_F^p \tilde{H}^{j}(\M,\C)_\lambda$$
$$\text{with $p=\lfloor n-\alpha \rfloor$, $\lambda=\exp(-2\pi i\alpha ),$}$$
where $\tilde{H}^j(\M,\C)_\lambda$ is the $\lambda$-eigenspace of the reduced cohomology under $T_s$ and $F$ is the Hodge filtration. It is known that $n_{f,\alpha}=0$ for $\alpha\not\in (0,n)$ (see \cite{BS-MVS}).
The {\it Hodge spectrum} of the germ $f$ is the fractional Laurent polynomial
$$Sp(f):=\sum_{\alpha \in \Q} n_{f,\alpha}t^\alpha.$$

\subsection{Spectrum of homogeneous polynomials} \label{homo}
Assume that $f$ is homogeneous with degree $d$. Then we can consider the divisor $Z\subset \bP^{n-1}=:Y$ defined by $f$. Let $\rho : \tilde{Y}\rightarrow Y$ be an embedded resolution of $Z$ inducing an isomorphism over $Y\backslash Z$. We have a divisor $\tilde{Z}:=\rho ^*Z$ with normal crossing on $\tilde{Y}$. Set $\tilde{Z}=\sum_{V\in J}m_V E_V$ where $E_V$ are the irreducible components with multiplicity $m_V$. Let $\tilde{H}$ be the total transform of a general hyperplane $H$ of $Y$. Then the eigenvalues of the monodromy are $d$-th roots of unity (see \cite{Bu-HH}-4) and we have the following formula for spectrum multiplicity (see \cite{Bu-S}-1.5). Note that the formula holds only on $\alpha \in (0,n)$ by \cite{BS-MVS}.

\begin{prop}\label{neuler} For $\alpha =n-p-\frac{i}{d}\in(0,n)$ with $p\in\Z$ and $i\in[0,d-1]\cap \Z$
	\begin{equation}\label{euler}
	n_{f,\alpha}=(-1)^{p-n+1} \X\left(\tilde{Y},\Omega_{\tilde{Y}} ^p (\log \tilde{Z})\bigotimes_{\Sf_{\tilde{Y}}} \Sf_{\tilde{Y}} \left( -i\tilde{H}+\sum_{V\in J} \lfloor im_V /d\rfloor E_V \right) \right),
	\end{equation}
where $\lfloor \cdot \rfloor$ is floor. 
\end{prop}

Using Hirzebruch-Riemann-Roch, we can calculate $n_{f,\alpha}$.\\

\begin{cor}\label{HHR}Let $\E_{i,p}:=\Omega_{\tilde{Y}} ^p (\log \tilde{Z})\bigotimes_{\Sf_{\tilde{Y}}} \Sf_{\tilde{Y}} \left( -i\tilde{H}+\sum_{V\in J} \lfloor i m_V /d\rfloor E_V \right)$. Then
\begin{equation}
n_{f,\alpha}=(-1)^{p-n+1}(ch(\E_{i,p})\cdot td(\tilde{Y}))_{n-1},
\end{equation}
where $ch(\E_{i,p})$ is the Chern character of $\E_{i,p}$ and $td(\tilde{Y})$ is the Todd class of the tangent bundle $T\tilde{Y}$.
\end{cor}

Now we use the intersection theory to calculate $ch(\E_{i,p})$ and $td(TX)$. We will denote $c(X):=c(TX)$, $ch(X):=ch(TX)$, and $td(X):=td(TX)$ for the tangent bundle $TX$ of a variety X. 

\subsection{Calculation from intersection theory}\label{ch}
Let $A:=\Omega_{\tilde{Y}}^1 (\log \tilde{Z})$ and $U_i:=\Sf_{\tilde{Y}} \left( -i\tilde{H}+\sum_{V\in J} \lfloor im_V /d\rfloor E_V \right)$. Then 
$$ch(\E_{i,p})=ch(\wedge ^p A)\cdot ch(U_i).$$ 
Thus, we have 
\begin{equation}\label{chchtd}
n_{f,\alpha}=(-1)^{p-n+1}(ch(\wedge ^p A)\cdot ch(U_i)\cdot td(\tilde{Y}))_{n-1}.
\end{equation}
Chern classes calculate the Chern character and Todd classes. For a given vector bundle $E$ of rank $r$, the following is well known (see \cite{Fu}-3.2),
\begin{equation}\label{generalch}
ch(E)=r+c_1(E)+\frac{1}{2}\left(c_1(E)^2-2c_2(E)\right)+\frac{1}{6}\left(c_1(E)^3 -3c_1(E)c_2(E)+3c_3(E)\right)+\cdots,\\
\end{equation}
and
\begin{equation}\label{td}
td(E)=1+\frac{1}{2}c_1(E)+\frac{1}{12}\left(c_1(E)^2+c_2(E)\right)+\frac{1}{24}\left(c_1(E)c_2(E)\right)+\cdots.
\end{equation}
 
Hence, we need to calculate $c(\wedge ^p A)$ to get $ch(\wedge ^p A)$. It can be calculated from the following formula. Let $A$ have rank $r$ and write the Chern polynomial $c_t(A)=\prod_{i=1}^r(1+x_i t)$ where $x_i$ are formal symbols. Then we have 
$$c_t(\wedge^p A)=\prod_{1\leq {i_1}<\cdots<{i_p}\leq r}\left(1+(x_{i_1}+\cdots+x_{i_p})t\right).$$

We should calculate $c(A)=c(\Omega_{\tilde{Y}}^1(\log\tilde{Z}))$. We have the following short exact sequence
$$0\rightarrow \Omega_{\tilde{Y}}^1\rightarrow\Omega_{\tilde{Y}}^1(\log\tilde{Z})\rightarrow \bigoplus_{j\in J}\Sf_{\tilde{Y}}(E_V)\rightarrow 0.$$
This induces
\begin{equation}\label{log}
c(\Omega_{\tilde{Y}}^1(\log\tilde{Z}))=c(\Omega_{\tilde{Y}}^1)\prod_{j\in J}c(\Sf_{\tilde{Y}}(E_V)).
\end{equation}

Summarizing this section we should construct $\tilde{Y}$ explicitly to calculate $n_{f,\alpha}$ in the homogeneous cases.

\section{Spectrum of Hyperplane arrangements}\label{hyper}

Let $D$ be a hyperplane arrangement defined by $f :\bC^n\rightarrow \bC$ with $D_{l}(l\in\Lambda)$ the irreducible components of $D$. We say that $D$ is {\it central} if all the $D_{l}$ pass through the origin. Assume that $D$ is central. Hence, $f$ is homogeneous so that we can apply Corollary \ref{HHR}. We define the {\it intersection lattice} $\cS(D)$ as
$$\cS(D)=\{\cap_{l\in I}D_{l}\}_{I\subset\Lambda,I\neq\emptyset}.$$
\noindent
Each element in this set is called an {\it edge}. For $V\in\cS(D)$, define $\gamma(V):=\codim_{\C^n} V$. An edge is {\it dense} if the subarrangement of hyperplanes containing it is {\it indecomposable} (see the definition of decomposable in the introduction). Let $\cS(D)^{dense}$ be the set of dense edges. Set 
$$\cS(D)^{dense\geq 2}=\{V\in\cS(D)^{dense}|\gamma(V)\geq 2\}.$$  Let $D^{nnc}\subset D$ denote the complement of the subset consisting of normal crossing singularities. Set
$$\cS(D)^{nnc}=\{V\in\cS(D)|V\subset D^{nnc}\}.$$
For a set $\cS$ of edges, let
$$\cS^{(k)}=\{V\in\cS|\gamma(V)=k\}.$$
Notice that $\cS(D)^{nnc}$ includes $\cS(D)^{dense\geq 2}$ and they coincide at codimension $2$.

\subsection{Construction of $\tilde{Y}$}
We construct $\tilde{Y}$ using successive blow-ups (see \cite{Bu-S}-2 and \cite{BMT}-2).
Let $Y_0=Y=\bP^{n-1}$. For a vector space $V\subset X=\bC^n$, its corresponding subspace of $Y$ will be denoted by $\bP (V)$. Let $\cS$ be any of $\cS(D)$, $\cS(D)^{dense}$, and $\cS(D)^{nnc}$. There is a sequence of blow-ups $\rho_i: Y_{i+1}\rightarrow Y_i$ for $0\leq i<n-2$ whose center is the disjoint union of the proper transforms of $\bP(V)$ for $V\in \cS$ with $\dim \bP(V)=i$. Set $\tilde{Y}=Y_{n-2}$ with $\rho:\tilde{Y}\rightarrow Y$ the composition of the $\rho_i$. This is the canonical log resolution of $(\bP^{n-1},\ Z)$ from \cite{DC}-4 where $Z$ is the divisor defined by $f$. 

Note that $\cS=\cS(D)^{dense}$ gives the minimal log resolution (see \cite{BMT}-2) but is not \emph{stable under intersection} (i.e. $V\cap V'\in \cS$ if $V,V'\in\cS$).

\subsection{Cohomology of $\tilde{Y}$}\label{coh}
Let $\cS:=\cS(D)^{nnc}$. By \cite{Bu-S}-5.3 (see also \cite{DC}-5) the cohomology ring of $\tilde{Y}$ is described as
\begin{equation}
\bQ[e_V]_{V\in\cS}/I_\cS\tilde{\rightarrow}H^{\bullet}(\tilde{Y},\bQ)
\end{equation}
sending $e_V$ to $[E_V]$ for $V\neq 0$ and $e_0$ to $-[E_0]$, where $e_V$ are independent variables for $V\in \cS$ and $E_0$ is the total transform of a general hyperplane which was denoted by $\tilde{H}$. Moreover, the ideal $I_\cS$ is generated by 
\begin{equation}\label{ideal}
R_{V,W}=\begin{cases}e_V e_W&\text{if $V$,$W$ are incomparable,}\\
e_V \tilde{e}_W^{\gamma(W)-\gamma(V)}&\text{if $W\subsetneq$V,}\\
\end{cases}
\end{equation}
where $\tilde{e}_W:=\sum_{W'\subset W}e_{W'}$. Here $V,W,W'\in\cS \cup \{\bC^n\}$ and $e_{\bC^n}=1$.\\

The stability of $\cS(D)^{nnc}$ under intersection was used in \cite{Bu-S} to get (\ref{ideal}) from \cite{DC} by observing that the \emph{nested} condition from \cite{DC} always becomes linearly ordered by the inclusion relation.
From now we use $e_V$ and $e_0$ for $[E_V]$ and $-[E_0]$ in $H^{\bullet}(\tilde{Y},\bQ)$.

\subsection{Calculation of $c(\tilde{Y})$}\label{chern}
$\tilde{Y}$ was constructed by the successive blow-ups. By \cite{Bu} (see also \cite{Fu}-Example 15.4.2), we have a formula for the Chern class of $\tilde{Y}$ , $c(\tilde{Y})=\prod_{V\in\cS}F_V$, where
\begin{equation}
F_V=\begin{cases}(1+e_V-\tilde{e}_V)^{-\gamma(V)}(1+e_V)(1-\tilde{e}_V)^{\gamma(V)}&\text{if $V\neq0$},\\
(1-e_0)^n&\text{if $V=0$.}
\end{cases}
\end{equation}

\subsection{Duality on $\tilde{Y}$}\label{pf1}
	Let $U$ be a divisor on $\tilde{Y}$. $u:=[U]$ can be written as $u=u_0 e_0+\sum_{V\in \cS} u_V e_V\in H^2(\tilde{Y})$ where $u_0, u_V\in \bZ$. Set $\cF_p(U):=\Omega_{\tilde{Y}} ^{p} (\log \tilde{Z})\otimes\Sf_{\tilde{Y}}(U)$ and consider  a function $\mu_p : H^2(\tilde{Y})\rightarrow \bZ$ defined by $\mu_p(u):=(-1)^{p-n+1}\X(\tilde{Y},\cF_p(U))$ for each $p\in\bZ$. By Proposition \ref{neuler}, $n_{f,n-p-\frac{i}{d}}=\mu_p\left(ie_0+\sum_{V\in \cS} \lfloor im_V /d\rfloor e_V \right)$. 
Using Serre duality, we get the following property.
\begin{prop}\label{symmetry}
Assume  $f$ is a reduced hyperplane arrangement of degree $d$. 
Then we have the following for $i\in \{0,\cdots,d-1\}$ and $p\in \{0,\cdots,n-1\}$
$$n_{f,p+1-\frac{i}{d}}=\mu_{p}\left((d-i) e_0+\sum_{V\in \cS} (m_V-1-\lfloor i m_V/d\rfloor) e_V\right).$$

\end{prop}
\begin{proof}
By Serre duality we have $H^q(\tilde{Y},\cF_p(U))\cong H^{n-1-q}(\tilde{Y},\cF_{n-1-p}(-\tilde{Z}_{red}-U))^{\vee}$ using $\Omega_{\tilde{Y}} ^{p} (\log \tilde{Z})=\Omega_{\tilde{Y}} ^{n-1-p} (\log \tilde{Z})^{\vee}\otimes\omega_{\tilde{Y}}\otimes\Sf_{\tilde{Y}}(\tilde{Z}_{red})$(see \cite{EV}-6.8 (b)). Let $z\in H^2(\tilde{Y})$ be the element corresponding to $\tilde{Z}_{red}$.
From this duality we have the following 
\begin{equation}\label{dual}
\mu_p (u)=\mu_{n-1-p}(-z-u).
\end{equation}

Let $d_l(l\in \Lambda)\in H^2(\tilde{Y})$ be  the element corresponding to the strict transform of an irreducible component $D_{l}$ in the hyperplane arrangement $D=\cup_{l\in \Lambda}D_{l}$. We have relations $d_l=-\left(e_0+\sum_{V\subset D_{l}}e_V \right)$ and $z=\sum_{V\in\cS}e_V+\sum_{l\in\Lambda}d_l$. Thus,
$$z=\sum_{V\in \cS}(1-m_V) e_V-d e_0.$$
Plugging into (\ref{dual}) we have
$$\mu_p (u_0 e_0+\sum_{V\in \cS} u_V e_V)=\mu_{n-1-p}((d-u_0) e_0+\sum_{V\in \cS} (m_V-1-u_V) e_V).$$
This equality means 
$$n_{f,p+1-\frac{i}{d}}=n_{f,n-(n-1-p)-\frac{i}{d}}=\mu_{p}\left((d-i) e_0+\sum_{V\in \cS} (m_V-1-\lfloor i m_V/d\rfloor) e_V\right).$$
\end{proof}

\section{Proof of Theorem \ref{n4}}
Recall that $e_V$ and $e_0$ are the generators of $H^{\bullet}(\tilde{Y},\bQ)$.
Let $a_V$, $b_W$ and $c$ denote the $e_V$ for $V\in \cS^{(2)}$, $e_W$ for $W\in \cS^{(3)}$ and $e_0$ respectively (see \ref{coh}).\\

\subsection{Calculation of $H^{\bullet}(\tilde{Y})$}
According to Section \ref{coh}, we have generators $a_V$, $b_W$ and $c$ and the following relations in $H^{\bullet}(\tilde{Y},\bQ)$ by (\ref{ideal}):
\begin{equation}\label{rel}
a_V a_{V'}=b_W b_{W'}=b_W c=a_V b_W=0\ (V\neq V', W\neq W',W\not\subset V),
\end{equation}
$$a_V b_W^2=a_V c^2=b_W^2 c=b_W c^2 =a_V b_W c=0,$$
$$a_V^3=2\left(1- \sum_{W\subset V} 1\right)c^3,\ a_V^2 c=b_W^3=-c^3, c^4=0,$$
$$\ a_V b_W=-a_V c\ (W\subset V), \text{ and } a_V^2 b_W =c^3\ (W\subset V).$$

From relations (\ref{rel}) we get:
\begin{align}\label{sums}
\left(\sum_{V\in \cS^{(2)}}N_V a_V^s\right) \left(\sum_{V\in \cS^{(2)}}N'_V a_V^{s'}\right) &= \sum_{V\in \cS^{(2)}} N_V N'_V a_V^{s+s'},\\
\left(\sum_{W\in \cS^{(3)}}N_W b_W^s\right) \left(\sum_{W\in \cS^{(3)}}N'_W b_W^{s'}\right) &= \sum_{W\in \cS^{(3)}} N_W N'_W b_W^{s+s'},\notag
\end{align}
and
\begin{align*}
\left(\sum_{V\in \cS^{(2)}}N_V a_V^s\right) \left(\sum_{W\in \cS^{(3)}}N'_V b_W^{s'}\right) &= \sum_{V\in \cS^{(2)}}\sum_{W\subset V} N_V N'_W a_V^s b_W^{s'}\notag
\end{align*}
for any coefficients $N_V$, $N'_V$, $N_W$, $N'_W$ and any positive integers $s$ and $s'$.
These equalities are very useful for the calculations.

\subsection{Calculation of Chern classes}\label{chern4}
According to Section \ref{chern}, we have
\begin{align*}
c(\tilde{Y})&=(1-c)^4\prod_{W\in \cS^{(3)}}\left( (1+b_W)\left(\frac{1 -c- b_W}{1 -c}\right)^3\right)\\
	&\cdot\prod_{V\in \cS^{(2)}}\left( (1+a_V)\left(\frac{1 -c-\sum_{W\subset V}b_W -a_V}{1 -c-\sum_{W\subset V}b_W}\right)^2\right).
\end{align*}

Since $c(\tilde{Y})=c((\Omega_{\tilde{Y}}^1)^{\vee})$, it is enough to change the signs of all the generators in the formula above for the calculation of $c(\Omega_{\tilde{Y}}^1)$
\begin{align*}
c(\Omega_{\tilde{Y}}^1)&=(1+c)^4\prod_{W\in \cS^{(3)}}\left( (1-b_W)\left(\frac{1 +c+ b_W}{1 +c}\right)^3\right)\\
	&\cdot\prod_{V\in \cS^{(2)}}\left( (1-a_V)\left(\frac{1 +c+\sum_{W\subset V}b_W +a_V}{1 +c+\sum_{W\subset V}b_W}\right)^2\right).
\end{align*}

From equation (\ref{log}), we have 
$$c(\Omega_{\tilde{Y}}^1(\log \tilde{Z}))=c(\Omega_{\tilde{Y}}^1)\prod_{W\in \cS^{(3)}}\frac{1}{1-b_W}\prod_{V\in \cS^{(2)}}\frac{1}{1-a_V}\prod_{l\in \Lambda}\frac{1}{1-d_l},$$\\
where $d_l(l\in \Lambda)$ correspond to the strict transform of an irreducible component $D_{l}$ of hyperplane arrangement $D=\cup_{l\in \Lambda}D_{l}$. Recall $d_l=-\left(c+\sum_{W\subset D_{l}}b_W +\sum_{V\subset D_{l}}a_V \right)$.

We calculate each factor of the Chern classes above, namely $c(\tilde{Y})$, $c(\Omega_{\tilde{Y}}^1)$ and $c(\Omega_{\tilde{Y}}^1(\log \tilde{Z}))$ using the relations (\ref{rel}):
$$\prod_{V\in \cS^{(2)}} (1-a_V) = 1+\sum_{V\in \cS^{(2)}}(-a_V),$$
$$\prod_{W\in \cS^{(3)}}(1-b_W)=1+\sum_{W\in \cS^{(3)}}(-b_W),$$
$$(1+c)^4=1+4c+6c^2+4c^3,$$
$$\prod_{W\in \cS^{(3)}}\left(\frac{1 +c+ b_W}{1 +c}\right)^3=1+3\sum_{W\in \cS^{(3)}} b_W+3\sum_{W\in \cS^{(3)}} b_W^2+\sum_{W\in \cS^{(3)}} b_W^3,$$
\begin{align*}
\prod_{V\in \cS^{(2)}}\left(\frac{1 +c+\sum_{W\subset V}b_W +a_V}{1 +c+\sum_{W\subset V}b_W}\right)^2&=1+2\sum_{V\in \cS^{(2)}} a_V\\
	&+\sum_{V\in \cS^{(2)}} \left(a_V^2-2(1-\sum_{W\subset V} 1)a_V c\right)+\sum_{V\in \cS^{(2)}} a_V^3,
\end{align*}
and 
\begin{align*}
\prod_{l\in \Lambda}\frac{1}{1-d_l}&=\prod_{l\in \Lambda}\frac{1}{1+c+\sum_{W\subset D_{l}}b_W+\sum_{V\subset D_{l}}a_V}\\
	&=1-\left(\sum_{V\in \cS^{(2)}} m_V a_V+\sum_{W\in \cS^{(3)}} m_W b_W +dc\right)\\
	&+\left(\sum_{V\in \cS^{(2)}}\left(\binom{m_V+1}{2}a_V^2+\left(d+1-\sum_{W\subset V} (m_W+1)\right)m_Va_V c\right)\right.\\
	&+\left.\sum_{W\in \cS^{(3)}}\binom{m_W+1}{2}b_W^2+\binom{d+1}{2}c^2\right)\\
	&-\left(\sum_{V\in \cS^{(2)}} \left(2\binom{m_V+1}{3}-d\binom{m_V+1}{2}\right)-\sum_{W\in \cS^{(3)}}\binom{m_W+2}{3}+\binom{d+2}{3}\right.\\
	&-\left.\sum_{V\in \cS^{(2)}} \sum_{W\subset V}\left(2\binom{m_V+1}{3}-m_W\binom{m_V+1}{2}\right)\right)c^3,
\end{align*}
where  $m_V=\sum_{V\subset D_{l}} 1$ and $m_W=\sum_{W\subset D_{l}} 1$.
We need to do some combinatorics for $\prod_{l\in \Lambda}\frac{1}{1-d_l}$.\\

From these factors we give formulas for the Chern classes:
\begin{align*}
c(\Omega_{\tilde{Y}}^1)&=1+\left( \sum_{V\in \cS^{(2)}} a_V +\sum_{W\in \cS^{(3)}} 2b_W +4c \right)+  \left(\sum_{V\in \cS^{(2)}} \left(-a_V^2 +2a_V c\right)+6c^2\right)\\
	&+\left(\sum_{V\in \cS^{(2)}} 2 +\sum_{W\in \cS^{(3)}} 2 +4\right)c^3,
\end{align*}
\begin{align*}
c(\tilde{Y})&=1-\left( \sum_{V\in \cS^{(2)}} a_V +\sum_{W\in \cS^{(3)}} 2b_W +4c\right)+\left ( \sum_{V\in \cS^{(2)}} (-a_V^2 +2a_V c)+6c^2\right)\\
	&-\left(\sum_{V\in \cS^{(2)}} 2 +\sum_{W\in \cS^{(3)}} 2 +4\right)c^3,
\end{align*}
and
\begin{align*}	
c(\Omega_{\tilde{Y}}^1(\log \tilde{Z}))&=1-\left( \sum_{V\in \cS^{(2)}} (m_V-2)a_V +\sum_{W\in \cS^{(3)}} (m_W-3)b_W +(d-4)c\right)-\\
	&+\left ( \sum_{V\in \cS^{(2)}} \left(\binom{m_V-1}{2}a_V^2 +(m_V-2)\left(d-3-\sum_{W\subset V} (m_W-2)\right)a_V c\right)\right.\\
	&+\left.\sum_{W\in \cS^{(3)}}\binom{m_W-2}{2}b_W^2+\binom{d-3}{2}c^2\right)\\
	&-\left(\sum_{V\in \cS^{(2)}} \left(2\binom{m_V-1}{3}-(d-4)\binom{m_V-1}{2}\right)-\sum_{W\in \cS^{(3)}}\binom{m_W-1}{3}\right.\\
	&+\left.\binom{d-2}{3}-\sum_{V\in \cS^{(2)}}\sum_{W\subset V}\left(2\binom{m_V-1}{3}-(m_W-3)\binom{m_V-1}{2}\right)\right)c^3.
\end{align*}

\subsection{Calculation of the Chern character of line bundle}\label{chlb}
Let $U$ be a divisor on $\tilde{Y}$. Then its class $u:=[U]$ can be written as $u=\sum_{V\in \cS^{(2)}} u_V a_V +\sum_{W\in \cS^{(3)}} u_W b_W+u_0 c$.  Since $c(\Sf_{\tilde{Y}}(U))=1+u$ we have 
\begin{align*}
ch(\Sf_{\tilde{Y}}(U))&=1+\left(\sum_{V\in \cS^{(2)}} u_V a_V +\sum_{W\in \cS^{(3)}} u_W b_W+u_0c\right)\\
&+\frac{1}{2}\left(\sum_{V\in \cS^{(2)}} \left( u_V^2 a_V^2+2u_V\left(u_0-\sum_{W\subset V} u_W \right)a_V c \right)+\sum_{W\in \cS^{(3)}}u_W^2 b_W^2 + u_0^2 c^2\right)\\
&+\frac{1}{6}\left(\sum_{V\in \cS^{(2)}}u_V^2\left((2u_V-3u_0)+\sum_{W\subset V} (3u_W-2 u_V)\right)-\sum_{W\in \cS^{(3)}} u_W^3+u_0^3   \right) c^3
\end{align*}
by formula (\ref{generalch}). Here we used (\ref{sums}) to simplify the calculation.

\subsection{Calculation of $n_{f,\alpha}$}
Let $\mu_p(u)=(-1)^{p-n+1}(ch(\Omega_{\tilde{Y}}^p(\log \tilde{Z}))\cdot ch(U)\cdot td(\tilde{Y}))_{n-1}$ as before (see Section \ref{pf1}).
We have calculated $ch(\Sf_{\tilde{Y}}(U))$ in Section \ref{chlb}. We can also calculate $td(\tilde{Y})$, $ch(\Omega_{\tilde{Y}}^p(\log \tilde{Z}))$ for $p=0,1,2,3$ from Section \ref{ch} and \ref{chern4}.  Multiplying all these we get the following, using (\ref{sums}) to simplify the calculation:

 	\begin{align*}
		\mu_0(u)&=\binom{u_0-1}{3}-\sum_{W\in \cS^{(3)}} \binom{u_W}{3}-\sum_{V\in \cS^{(2)}}\left( (u_0-3)\binom{u_V}{2}-2\binom{u_V}{3}\right)\\
			&-\sum_{V\in \cS^{(2)}} \sum_{W\subset V}\left( 2 \binom{u_V}{3} -(u_W-2)\binom{u_V}{2}\right),\\
		\mu_1(u)&=(d-u_0-1)\binom{u_0-1}{2}-\sum_{W\in \cS^{(3)}}(m_W-u_W-1)\binom{u_W}{2}\\
			&-\sum_{V\in\cS^{(2)}}\left(u_{V}(m_V-u_{V}-1)(u_0-2) +(d-u_0-1-2(m_V-u_{V}-1))\binom{u_{V}}{2}\right)\\
			&+\sum_{V\in\cS^{(2)}}\sum_{W\subset V}\left( u_{V}(m_V-u_{V}-1)(u_{W}-u_{V})+(m_W-u_{W}-1)\binom{u_{V}}{2}\right),\\
		\mu_2(u)&=(u_0-1)\binom{d-u_0-1}{2}-\sum_{W\in \cS^{(3)}} u_{W}\binom{m_W-u_W-1}{2}\\
			&-\sum_{V\in\cS^{(2)}}\left((m_V-u_{V}-1)u_{V}(d-u_0-2) +(u_0-1-2u_{V})\binom{m_V-u_{V}-1}{2}\right)\\
			&+\sum_{V\in\cS^{(2)}}\sum_{W\subset V}\left( \dfrac{}{} (m_V-u_{V}-1)u_{V}((m_W-u_{W}-1)-(m_V-u_{V}-1))\right.\\
			&\left.+u_{W}\binom{m_V-u_{V}-1}{2}\right),\text{ and}\\
		\mu_3(u)&=\binom{d-u_0-1}{3}-\sum_{W\in \cS^{(3)}} \binom{m_W - u_W-1}{3}\\
			&-\sum_{V\in \cS^{(2)}}\left( (d-u_0-3)\binom{m_V-u_V-1}{2}-2\binom{m_V - u_V-1}{3}\right)\\
			&- \sum_{V\in \cS^{(2)}} \sum_{W\subset V}\left( 2 \binom{m_V-u_V-1}{3} -(m_W-u_W-3)\binom{m_V-u_V-1}{2}\right).
	\end{align*}

By formula (\ref{chchtd}), $n_{f,\alpha}=\mu_p\left(\left( ic +\sum_{W\in \cS^{(3)}} \lfloor i m_W /d\rfloor b_W+\sum_{V\in \cS^{(2)}} \lfloor i m_V /d\rfloor a_V \right)\right)$ with $\alpha =4-p-\frac{i}{d}\in (0,4)$.  From the relation $m_V-\lfloor im_V/d\rfloor=\lceil (d-i)m_V/d\rceil$ and substituting $d-i$ for $i$ for $\alpha\in(0,2]$, we get the formula in Theorem \ref{n4} for $\cS=\cS(D)^{nnc}$. 

\begin{rem}
Here we calculated all the $\mu_p$ without using Proposition \ref{symmetry}. We can and did use Proposition \ref{symmetry} to double-check the formulas for $\mu_p$. In other words, we can get the formulas for $n_{f,4-\frac{i}{d}}$ and $n_{f,3-\frac{i}{d}}$ from the formulas for  $n_{f,1-\frac{i}{d}}$ and $n_{f,2-\frac{i}{d}}$ respectively and vice versa. Also, all computations were double-checked by computer. The implementation of the symbolic computation is possible because of the relations (\ref{sums}).
\end{rem}

\subsection{Non-dense edges}\label{cS}
In this section, we will prove that $\cS$ in Theorem \ref{n4} can be replaced by any set of edges containing all the dense edges with codimension $\geq 2$ by showing the vanishing of all the terms which depend on edges in $\cS\backslash\cS(D)^{dense\geq 2}$. This implies Theorem \ref{n4}. Fix a set $\cS$ containing $\cS(D)^{dense\geq 2}$.

First of all, the edges in $\cS^{(1)}$ are not used in our formula. 
If an edge $V\in \cS^{(2)}$ is not dense the terms in $\eta_{0,i}(\langle u_{V}\rangle_{V\in \cS})$ and the terms in $\eta_{1,i}(\langle u_{V},v_{V}\rangle_{V\in \cS})$ depending on $V\in\cS^{(2)}$ vanish  since $u_V+v_V=1$ and  $u_V$ is $0$ or $1$ for any formula for $n_{f,\alpha}$ in Theorem \ref{n4}.

The terms $\eta_{0,i,W}$ in $\eta_{0,i}(\langle u_{V}\rangle_{V\in \cS})$ and the terms $\eta_{1,i,W}$ in $\eta_{1,i}(\langle u_{V},v_{V}\rangle_{V\in \cS})$ depending on $W\in\cS^{(3)}$ are
\begin{align*}
\eta_{0,i,W}&=- \binom{u_{W}}{3}-\sum_{\substack{V\supset W\\ V\in \cS^{(2)}}}\left( 2 \binom{u_{V}}{3} -(u_{W}-2)\binom{u_{V}}{2}\right)\text{ and }\\
\eta_{1,i,W}&=-v_{W}\binom{u_{W}}{2}+\sum_{\substack{V\supset W\\ V\in \cS^{(2)}}}\left( u_{V}v_{V}(u_{W}-u_{V})+v_{W}\binom{u_{V}}{2}\right).\\
\end{align*}
If $V\in\cS^{(2)}$ in $\eta_{0,i,W}$ and $\eta_{1,i,W}$ (i.e. $V\supset W$) is not dense, the terms depending on $V$ vanish since $m_V=2$. Thus, we may assume that $V\in\cS^{(2)}$ in $\eta_{0,i,W}$ and $\eta_{1,i,W}$ are dense edges with codimension $2$. 

When an edge $W\in\cS^{(3)}$ is not dense, we have two possibilities: either $W\notin\cS(D)^{nnc}$ or $W\in\cS(D)^{nnc}\backslash\cS(D)^{dense\geq 2}$. If $W\notin\cS(D)^{nnc}$, then $m_W=3$ and we do not have dense edge $V\in\cS^{(2)}$ in $\eta_{0,i,W}$ and $\eta_{1,i,W}$. Also, $m_W=3$ implies that $u_W+v_W=2$ and  $u_W$ is $0$,$1$ or $2$ for any formula for $n_{f,\alpha}$ in Theorem \ref{n4}. Hence, $\eta_{0,i,W}=\eta_{1,i,W}=0$. In the case of $W\in\cS(D)^{nnc}\backslash\cS(D)^{dense\geq 2}$, we have exactly one codimension 2 dense edge $V_W\in\cS^{(2)}$ such that $W\subset V_W$ since the subarrangement of hyperplanes containing $W$ is decomposable. Moreover, $m_W=m_{V_W}+1$. Thus,
\begin{align*}
\eta_{0,i,W}&=-\binom{u_{W}}{3}-\left( 2 \binom{u_{V_W}}{3} -(u_{W}-2)\binom{u_{V_W}}{2}\right)\text{ and }\\
\eta_{1,i,W}&=-v_{W}\binom{u_{W}}{2}+\left( u_{V_W}v_{V_W}(u_{W}-u_{V_W})+v_{W}\binom{u_{V_W}}{2}\right).\\
\end{align*}
For any formula for $n_{f,\alpha}$ in Theorem \ref{n4} we have $u_W+v_W=m_W-1$ and $u_{V_W}+v_{V_W}=m_{V_W}-1$.  From $m_W=m_{V_W}+1$ we have only two possibilities for each formula for $n_{f,\alpha}$: either $u_W=u_{V_W}$ or $u_W=u_{V_W}+1$. For the first case, $v_{W}=m_W-1-u_W=m_{V_W}-u_{V_W}=v_{V_W}+1$.
For the second case, $v_{W}=m_W-1-u_W=m_{V_W}-u_{V_W}-1=v_{V_W}$.
Both cases make $\eta_{0,i,W}=0$ and $\eta_{1,i,W}=0$.

Hence, only the terms depending on dense edges can survive. This proves Theorem \ref{n4}.\\
\rightline{$\Box$}

\section{Proof of Corollaries}
We need the Thom-Sebastiani formula for Corollaries \ref{n3} and \ref{n2}. Hyperplane arrangements have non-isolated singularities but the formula still holds in our case (see \cite{Kul}-II (8.10.6)). Here we state a special case.

\begin{lem}\label{thom}
Assume that $f:(\bC^n,\ 0)\rightarrow (\bC,\ 0)$ can be written as $f(x_1,\cdots,x_n)=g(x_1,\cdots,x_m)$ for $g:(\bC^m,\ 0)\rightarrow (\bC,\ 0)$. Then $Sp(f)=(-t)^{n-m}Sp(g)$.
\end{lem}

\subsection{Proof of Corollary \ref{n3}} \label{3}
Consider $f(x_1,x_2,x_3)$ with degree $d$ in $\bC^3$. Let $f(x_1,x_2,x_3)=g(x_1,x_2,x_3,x_4)$. We will calculate $Sp(g)$. First, we assume that $\cS^{(3)}\neq\emptyset$. The hyperplane arrangement $g$ has only one codimension $3$ dense edge $W=\{x_1=x_2=x_3=0\}$. Moreover the multiplicity $m_W$ is $d$. Thus, $\lceil im_{W}/d \rceil-1=i-1$, $\lfloor (d-i) m_{W}/d \rfloor=d-i$, $\lfloor i m_{W}/d \rfloor=i$, and $\lceil (d-i)m_{W}/d \rceil-1=d-i-1$. Using Theorem \ref{n4} we get a formula for $n_{g,\alpha}$.

When $\cS^{(3)}=\emptyset$, $\cS^{(2)}$ has only one element, $V$, or no elements. In the case that $\cS^{(2)}=\{V\}$, $m_V=d$ or $d-1$. The calculation from Theorem \ref{n4} for both cases coincides with the calculation from the formula above. If $\cS^{(2)}=\emptyset$, the hyperplane arrangement is a generic case with $1\leq d\leq 3$ which also satisfies the formula above.

The set $\cS^{(2)}$ has a one-to-one correspondence to the set $\cS$ of codimension 2 dense edges of $f$. This proves Corollary \ref{n3} by Lemma \ref{thom}.\\
\rightline{$\Box$}

\subsection{Proof of Corollary \ref{n2}}
Consider $f(x_1,x_2)$ with degree $d$ in $\bC^2$. Let $f(x_1,x_2)$ $=h(x_1,x_2,x_3,x_4)$. We will calculate $Sp(h)$. The hyperplane arrangement $h$ has no codimension $3$ dense edge. First, we assume that $\cS^{(2)}\neq\emptyset$. The hyperplane arrangement $h$ has only one codimension $2$ dense edge $V=\{x_1=x_2=0\}$. Moreover the multiplicity $m_V$ is $d$. Thus, $\lceil im_{V}/d \rceil-1=i-1$, $\lfloor (d-i) m_{V}/d \rfloor=d-i$, $\lfloor i m_{V}/d \rfloor=i$, and $\lceil (d-i)m_{V}/d \rceil-1=d-i-1$. Using Theorem \ref{n4} we get a formula for $n_{h,\alpha}$.
If $\cS^{(2)}=\emptyset$, then $d=1$ or $2$. The calculation from Theorem \ref{n4} coincides with the calculation from the formula above. This proves Corollary \ref{n2} by Lemma \ref{thom}.\\
\rightline{$\Box$}

\subsection{Proof of Corollary \ref{3:1}}
Notice that $f_2(x_4)=x_4$ and $\cS^{(2)}$ is the set of codimension 2 dense edges of subarrangement $f_1(x_1,x_2,x_3)$ in $\bC^4$. First, we assume $\cS^{(3)}\neq\emptyset$. We have only one codimension $3$ dense edge $W_\infty=\{x_1=x_2=x_3=0\}$ with the multiplicity $m_{W_\infty}=d-1$. Applying Theorem \ref{n4} to these we get,
$$\eta_{0,i}(\langle u_{V}\rangle_{V\in \cS})=\binom{i-1}{3}-\binom{u_{W_\infty}}{3}+\sum_{V\in\cS^{(2)}}(u_{W_\infty}-i+1)\binom{u_V}{2}+\delta_{0,i}\text{ and}$$
\begin{align*}
\eta_{1,i}&(\langle u_{V},v_{V}\rangle_{V\in \cS})=(d-i-1)\binom{i-1}{2}-v_{W_\infty}\binom{u_{W_\infty}}{2}\\
&+\sum_{V\in\cS^{(2)}}\left(u_{V}v_{V}(u_{W_\infty}-u_{V}-i+2) +(v_{W_\infty}+2v_{V}-d+i+1)\binom{u_{V}}{2}\right).
\end{align*}

We calculate $n_{f,\alpha}$ for $\alpha\in(0,2]$. If $i\in\{1,\cdots,d-1\}$, then $u_{W_\infty}=\lceil i(d-1)/d \rceil-1=i-1$ and $v_{W_\infty}=\lfloor (d-i) (d-1)/d \rfloor=d-i-1$. We get 
$n_{f,i/d}= n_{f,1+i/d}=0.$
 If $i=d$, then $u_V=m_{V}-1$, $v_V=0$, $u_{W_\infty}=d-2$ and $v_{W_\infty}=0$. We get the formulars for $n_{f,1}$ and $n_{f,2}$.

We calculate $n_{f,\alpha}$ for $\alpha\in(2,4]$. If $i\in\{1,\cdots,d-1\}$, then $u_{W_\infty}=\lfloor im_{W_\infty}/d \rfloor=i-1$ and $v_{W_\infty}=\lceil (d-i) m_{W_\infty}/d \rceil-1=d-i-1$. We get 
$n_{f,3-i/d}= n_{f,4-i/d}=0.$
 If $i=0$, then $u_V=\lfloor im_{V}/d \rfloor=0$, $v_V=\lceil (d-i) m_{V}/d \rceil-1=m_{V}-1$, $u_{W_\infty}=\lfloor im_{W_\infty}/d \rfloor=0$ and $v_{W_\infty}=\lceil (d-i) m_{W_\infty}/d \rceil-1=d-2$. We get $n_{f,3}=d-1$ and $n_{f,4}=0.$

When $\cS^{(3)}=\emptyset$, $\cS^{(2)}$ has only one element $V$ or no elements. In the case that $\cS^{(2)}=\{V\}$, $m_V=d-1$ or $d-2$. The calculation from Theorem \ref{n4} for both cases coincides with the calculation from Corollary \ref{3:1}. If $\cS^{(2)}=\emptyset$, the hyperplane arrangement is a generic case with $2\leq d\leq 4$ since $f_1$ and $f_2$ are not constant. The calculation from Theorem \ref{n4} for both cases coincides with the calculation from Corollary \ref{3:1} for any of $2\leq d\leq 4$.
This proves Corollary \ref{3:1}.\\
\rightline{$\Box$} 

\subsection{Proof of Corollary \ref{2:2}}
In the case that $1\leq s_1\leq 2$ and $1\leq s_2\leq 2$, we can apply Corollary \ref{generic} and get the result of Corollary \ref{2:2}. We may assume $s_1 > 2$. If $s_2=1$, we use the formula for $n_{g,\alpha}$ in Section \ref{3} with $\cS^{(2)}=\{V=\{x_1=x_2=0\}\}$ and $s_1=m_V=d-1$. This proves Corollary \ref{2:2} for this case. If $s_2=2$, we can apply Corollary \ref{3:1} since $f_2$ can be written as $x_3 x_4$ after a suitable change of coordinates. In this case $\cS^{(2)}=\{V=\{x_1=x_2=0\}\}$ and $s_1=m_V=d-2$. The result satisfies Corollary \ref{2:2}. Thus, we assume that $s_1>2$ and $s_2>2$. We have only two dense edges in $\cS^{(2)}$, $V_1=\{x_1=x_2=0\}\text{ and }V_2=\{x_3=x_4=0\}$. Their multiplicity $m_{V_1}$ and $m_{V_2}$ are $s_1$ and $s_2$ respectively. Notice that $\cS^{(3)}=\emptyset$. Applying Theorem \ref{n4} to these we get,
\begin{align*}
\eta_{0,i}&(\langle u_{V}\rangle_{V\in \cS})=\binom{i-1}{3}-\sum_{j=1,2}\left( (i-3)\binom{u_{V_j}}{2}-2\binom{u_{V_j}}{3}\right)+\delta_{0,i}\text{ and }\\
\eta_{1,i}&(\langle u_{V},v_{V}\rangle_{V\in \cS})=(d-i-1)\binom{i-1}{2}\\
&-\sum_{j=1,2}\left(u_{V_j}v_{V_j}(i-2) +(d-i-1-2v_{V_j})\binom{u_{V_j}}{2}\right).
\end{align*}

First, we calculate $n_{f,\frac{i}{d}}$ and $n_{f,1+\frac{i}{d}}$ with $m_{V_1}=s_1$, $m_{V_1}=s_1$, and $s_1+s_2=d$. If $gcd(s_1,s_2)=1$ and $i\in\{1,\cdots,d-1\}$, then the common factor $\lceil i s_1/d\rceil+\lceil i s_2/d\rceil-i-1$ of $n_{f,\frac{i}{d}}$ and $n_{f,1+\frac{i}{d}}$ vanishes because $\lceil i s_2/d \rceil=\lceil i (d-s_1)/d \rceil=i-\lfloor i s_1/d \rfloor$ and $gcd(s_1,d)=1$. Hence, we have $n_{f,\frac{i}{d}}=0\text{ and }n_{f,1+\frac{i}{d}}=0$.
If $gcd(s_1,s_2)=1$ and $i=d$, then we get $n_{f,1}=(s_1-1)(s_2-1)\text{ and } n_{f,2}=1-s_1 s_2$.

Similarly, we calculate $n_{f,3-\frac{i}{d}}$ and $n_{f,4-\frac{i}{d}}$. If $gcd(s_1,s_2)=1$ and $i\in\{1,\cdots,d-1\}$, then the common factor $\lfloor i s_1/d\rfloor+\lfloor i s_2/d\rfloor-i+1=\lfloor i s_1/d \rfloor-\lceil i s_1/d\rceil+1$ of $n_{f,3-\frac{i}{d}}$ and $n_{f,4-\frac{i}{d}}$ vanishes. Therefore, we have $n_{f,3-\frac{i}{d}}=0\text{ and }n_{f,4-\frac{i}{d}}=0$.
If $gcd(s_1,s_2)=1$ and $i=0$, then we get $n_{f,3}=s_1 + s_2-1\text{ and } n_{f,4}=0$. This proves Corollary \ref{2:2}.\\
\rightline{$\Box$}


\begin{thebibliography}{EMS}


\bibitem[1]{Bu-S} N. Budur and M.Saito, Jumping coefficients and spectrum of a hyperplane arrangement. Math. Ann. 347 (2010), no.3, 545--579. 17, 23.

\bibitem[2]{BS-MVS} N. Budur and M.Saito, Multiplier ideals, $V$-filtration, and spectrum. J. Alg. Geom. 14 (2005), 269--119.


\bibitem[3]{BMT} N. Budur, M. Musta\c{t}\u{a}, and Z. Teitler, The monodromy conjecture for hyperplane arrangements. Geom. Dedicata. 153 (2011), no. 1, 131-–137.

\bibitem[4]{Bu-HS} N. Budur, On Hodge spectrum and multiplier ideals.  Math. Ann. 327  (2003),  no. 2, 257--270.

\bibitem[5]{Bu}N. Budur, Jumping numbers of hyperplane arrangements. Comm. Algebra. Vol. 38, (2010), 1122--1136.

\bibitem[6]{Bu-HH}N. Budur, Hodge spectrum of hyperplane arrangements. arXiv:0809.3443(unpublished).

\bibitem[7]{DC}C. De Concini and C. Procesi, Wonderful models of subspace arrangements. Selecta Math.(N.S) 1 (1995), 459--494.

\bibitem[8]{Da}A. Dimca, Tate properties, polynomial-count varieties, and monodromy of hyperplane arrangements. Nagoya Math. J. Vol. 206, (2012), 75--97.

\bibitem[9]{EV}H. Esnault and E. Viehweg, Lectures on vanishing theorems. DMV Seminar, 20. Birkh\"auser Verlag, Basel, (1992).

\bibitem[10]{Fu}W. Fulton, Intersection Theory. Springer, Berlin, (1984).

\bibitem[11]{Kul}V. Kulikov, Mixed Hodge structures and singularities. Cambridge Univ. Press, Cambridge, (1998) xxii+186.

\bibitem[12]{P-V}P.D. Gonz\'{a}lez P\'{e}rez and M. Gonz\'{a}lez Villa, Motivic Milnor fiber of a quasi-ordinary hypersurface. Reine Angew. Math. Vol. 2014, (2012), 159–-205.

\bibitem[13]{St}C.A.M. Peters and J.H.M. Steenbrink, Mixed Hodge Structures. Ergeb. Math. Grenzgeb. 3. Folge Vol. 52. Springer-Verlag Berlin Heidelberg (2008).

\bibitem[14]{Sa8}M. Saito, Multiplier ideals, b-function, and spectrum of a hypersurface singularity. Compos. Math. {143} (2007), 1050–-1068.



\end{thebibliography}
\end{document}